\documentclass{article}
\usepackage{inconsolata}
\usepackage{authblk}
\usepackage[margin=1in]{geometry}
\usepackage{newtxtext,newtxmath}   
\usepackage{microtype}

\usepackage{amsthm}
\usepackage{amsmath,amssymb,mathtools}
\usepackage{enumitem}
\usepackage[colorlinks=true,linkcolor=blue,citecolor=blue,urlcolor=blue]{hyperref}

\usepackage{titlesec}
\titleformat{\section}{\normalfont\scshape}{\thesection.}{0.6em}{}
\titleformat{\subsection}{\normalfont\itshape}{\thesubsection.}{0.6em}{}
\titlespacing*{\section}{0pt}{2.0ex plus .6ex minus .3ex}{1.0ex plus .3ex}
\titlespacing*{\subsection}{0pt}{1.4ex plus .5ex minus .3ex}{0.8ex plus .2ex}

\numberwithin{equation}{section}
\newtheorem{theorem}{Theorem}[section]
\newtheorem{lemma}[theorem]{Lemma}
\theoremstyle{remark}
\newtheorem{remark}[theorem]{Remark}

\newcommand{\dd}{\mathrm{d}}

\title{Self-Generated Measures and the Centroid Rigidity of Power Laws}

\author[1]{Vincent E.~Coll, Jr.\thanks{Corresponding author. Email: \texttt{vec208@lehigh.edu}. ORCID: 0000-0002-5775-5522}}
\author[2]{Lee B.~Whitt}

\affil[1]{Department of Mathematics, Lehigh University, Bethlehem, PA 18015, USA}
\affil[2]{Northrop Grumman — Distinguished Technical Fellow (retired), San Diego, CA\\ \texttt{lee.barlow.whitt@hotmail.com}}

\date{}

\begin{document}
\maketitle


\begin{abstract}
We revisit a classical calculus computation—the centroid of the subgraph of a function on $[0,a]$—and show that it hides a rigidity theorem. 
Let $f:(0,\infty)\to(0,\infty)$ be $C^2$ with $f(0^+)=0$, and set
\[
\bar x(a)=\frac{\int_0^a x f(x)\,\dd x}{\int_0^a f(x)\,\dd x},\qquad
\bar y(a)=\frac12\,\frac{\int_0^a f(x)^2\,\dd x}{\int_0^a f(x)\,\dd x}.
\]
We prove that the \emph{Geometric Scaling Property}
\[
\bar y(a)=\lambda\,f(\bar x(a))\qquad(\forall a>0)
\]
holds \emph{if and only if} $f(x)=A x^{p}$ with $A>0$ and $p>0$, and for power laws the optimal constant is
\[
\lambda=\frac{p+1}{2(2p+1)}\left(\frac{p+2}{p+1}\right)^{p}.
\]
After a scale-free normalization, the argument is probabilistic: under the self-generated probability measure on $[0,a]$ with density 
$f(x)\big/\!\int_0^a f(t)\,\dd t$, one has $\bar x(a)=\mathbb{E}[X_a]$ and $\bar y(a)=\tfrac12\,\mathbb{E}[f(X_a)]$, so GSP becomes an equality in expectation across all truncation scales. Differentiation in $a$ yields a weighted-mean identity for the elasticity $E(x)=x f'(x)/f(x)$; a second differentiation forces a vanishing-variance principle that makes $E$ constant, hence $f$ a pure power and the stated $\lambda$ follows. The proof uses no asymptotics and extends to $f\in C^1(0,\infty)$ with locally Lipschitz elasticity.
\end{abstract}

\section{Introduction and motivation}
Centroid formulas from first-year calculus suggest asking whether a persistent relation can hold between the centroid height and the function value sampled at the centroid abscissa (see, e.g., \cite{ApostolCalculus}). Precisely, for $f>0$ with $f(0^+)=0$, define
\[
\bar x(a)=\frac{\int_0^a x f(x)\,\dd x}{\int_0^a f(x)\,\dd x},\qquad
\bar y(a)=\frac12\,\frac{\int_0^a f(x)^2\,\dd x}{\int_0^a f(x)\,\dd x}.
\]
We call
\[
\bar y(a)=\lambda\,f(\bar x(a))\qquad(\forall a>0)
\]
the \emph{Geometric Scaling Property} (GSP). Despite its elementary appearance, GSP is rigid: we show it forces $f$ to be a power law and compute $\lambda$ explicitly. Our argument is scale-free and probabilistic.

\noindent\textit{Literature.}
Our WM/Var step can be viewed as an \emph{$L^2$ equality–case} (variance–zero / Cauchy–Schwarz) phenomenon (cf.\ \cite{HLP} for equality cases in quadratic forms) and as a functional–equation constraint on logarithmic scales (cf.\ \cite{Kuczma}); see also \cite{BGT} for the role of logarithmic derivatives in regular variation.

\noindent\textit{Functional-equation viewpoint.}
Equation~\eqref{eq:GSP-shape} is a functional equation on logarithmic scales; our proof proceeds via an equality–case derivation.

\subsection*{Provenance (and why this is trickier than it looks)}
This note grew out of a first–year calculus discussion about centroids. We asked a naïve question:
could the centroid height of the subgraph on $[0,a]$ be proportional to the function value sampled at
the centroid abscissa, for \emph{every} truncation length $a$? On the board this looks like harmless
bookkeeping—two familiar centroid formulas and a single proportionality constant. The surprise is that
this “calculus identity” quietly enforces a \emph{global} structure on $f$: when written in scale–free
form, the relation becomes an equality in expectation for a naturally self–generated probability measure.
Differentiating with respect to the truncation scale $a$ turns the question into one about the \emph{elasticity}
$E(x)=x f'(x)/f(x)$; a second differentiation reveals a \emph{variance–zero (Cauchy–Schwarz equality)} case. The only way this can
persist across all $a$ is if the elasticity is constant—so $f$ is a pure power law, and the proportionality
constant is forced as well.

In short, a classroom computation leads to a rigidity theorem: an apparently benign proportionality at the
\emph{centroid} pins down the \emph{shape} of the entire function. The argument stays elementary, but the
difficulty is conceptual—one is comparing quantities formed at different places (the integrals over $[0,a]$)
and at different scales (the function sampled at the moving point $\bar x(a)$), and only after recasting the
problem probabilistically does the hidden symmetry become visible.

\section{Setup and notation}\label{sec:setup}
Let $f\in C^{2}(0,\infty)$ with $f(x)>0$ for $x>0$ and $f(0^+)=0$. Define
\[
F(a)=\int_0^{a} f(x)\,\dd x,\qquad
H(a)=\int_0^{a} x\,f(x)\,\dd x,\qquad
G(a)=\int_0^{a} f(x)^2\,\dd x.
\]
The centroid of the subgraph $\{(x,y):0\le x\le a,\ 0\le y\le f(x)\}$ is
\[
\bar x(a)=\frac{H(a)}{F(a)},\qquad
\bar y(a)=\frac{G(a)}{2F(a)}.
\]

\subsection*{Notation (all symbols used)}
\begin{itemize}[leftmargin=1.5em]
\item $f:(0,\infty)\to(0,\infty)$: base function; $f(0^+)=0$; smoothness as stated.
\item $F,H,G$: primitives $F(a)=\int_0^a f$, $H(a)=\int_0^a x f$, $G(a)=\int_0^a f^2$.
\item $\bar x(a),\bar y(a)$: centroid coordinates of the subgraph on $[0,a]$.
\item $g_a(s):=\dfrac{f(as)}{f(a)}$ for $s\in[0,1]$ (shape profile at scale $a$).
\item $A(a):=\dfrac{F(a)}{a f(a)}$, \ $B(a):=\dfrac{H(a)}{a^2 f(a)}$, \ $C(a):=\dfrac{G(a)}{a f(a)^2}$.
\item $\theta(a):=\dfrac{B(a)}{A(a)}\in(0,1)$ (scale-free centroid abscissa).
\item $E(x):=\dfrac{x\,f'(x)}{f(x)}$ (elasticity; logarithmic derivative of $f$).
\item $\mu_a$: self-generated probability on $[0,a]$, $\mu_a(\dd x)=\dfrac{f(x)}{F(a)}\mathbf{1}_{[0,a]}(x)\,\dd x$.
\item $X_a\sim\mu_a$: random variable with $\mathbb{E}[X_a]=\bar x(a)$ and $\mathbb{E}[f(X_a)]=2\bar y(a)$.
\item $D(a):=\int_0^1 (s-\theta(a))^2 g_a(s)\,\dd s>0$ (quadratic weight normalizer).
\end{itemize}

Fix $a>0$ and define the \emph{rescaled shape}
\[
 g_a(s):=\frac{f(as)}{f(a)}\qquad (0\le s\le 1),
\]
and the \emph{elasticity}
\[
 E(x):=\frac{x\,f'(x)}{f(x)}.
\]
The function $E(x)$ is called the \emph{elasticity} of $f$.\footnote{Elasticity is familiar in economics as the percentage change in $f$ per percentage change in $x$ (cf.\ \cite{VarianMicro}); analytically it is the logarithmic derivative, as used in regular variation \cite{BGT}.}

A change of variables $x=as$ gives the normalized moments
\begin{equation}\label{eq:ABC}
A(a):=\int_0^1 g_a(s)\,\dd s=\frac{F(a)}{a f(a)},\quad
B(a):=\int_0^1 s\,g_a(s)\,\dd s=\frac{H(a)}{a^2 f(a)},\quad
C(a):=\int_0^1 g_a(s)^{2}\,\dd s=\frac{G(a)}{a f(a)^{2}}.
\end{equation}
Set
\[
 \theta(a):=\frac{B(a)}{A(a)}\in(0,1),\qquad \bar x(a)=a\,\theta(a).
\]
Two basic differential identities (chain and quotient rules), valid for $s\in(0,1]$, are
\begin{equation}\label{eq:ga-derivs}
 g_a(1)=1,\qquad
 g_a'(s)=\frac{E(as)}{s}\,g_a(s),\qquad
 \partial_a g_a(s)=\frac{g_a(s)}{a}\,\big(E(as)-E(a)\big).
\end{equation}
These identities are intended for $s\in(0,1]$; the factor $1/s$ in $g_a'(s)$ precludes evaluation at $s=0$.

\section{Probabilistic reformulation}
For each $a>0$, define a probability measure $\mu_a$ on $[0,a]$ by
\[
\mu_a(\dd x)=\frac{f(x)}{F(a)}\,\mathbf{1}_{[0,a]}(x)\,\dd x,\qquad F(a)=\int_0^a f(x)\,\dd x,
\]
and let $X_a\sim\mu_a$. Then
\[
\bar x(a)=\mathbb{E}[X_a],\qquad \bar y(a)=\tfrac12\,\mathbb{E}[f(X_a)].
\]
Thus GSP is the equality-in-expectation
\begin{equation}\label{eq:GSP-prob}
\tfrac12\,\mathbb{E}[f(X_a)] \;=\; \lambda\, f\!\big(\mathbb{E}[X_a]\big)\qquad(\forall a>0).
\end{equation}
Passing to scale-free variables $s=x/a$, set
\[
g_a(s):=\frac{f(as)}{f(a)},\qquad \nu_a(\dd s):=\frac{g_a(s)}{A(a)}\,\mathbf{1}_{[0,1]}(s)\,\dd s,
\]
so that if $S_a\sim\nu_a$ then $\mathbb{E}[S_a]=\theta(a)$ and $\mathbb{E}[g_a(S_a)]=C(a)/A(a)$. In these terms, \eqref{eq:GSP-prob} is equivalent to the scale-free identity
\begin{equation}\label{eq:GSP-shape}
 \frac12\,\frac{C(a)}{A(a)}=\lambda\,g_{a}\!\big(\theta(a)\big)\qquad(\forall a>0).
\end{equation}

\section{Main theorem}\label{sec:main}
\begin{theorem}\label{thm:GSP}
Assume the \emph{Geometric Scaling Property} \eqref{eq:GSP-shape}. Then there exist $A>0$ and $p>0$ such that $f(x)=A\,x^{p}$ for all $x>0$, and
\[
 \lambda=\frac{p+1}{2(2p+1)}\!\left(\frac{p+2}{p+1}\right)^{\!p}.
\]
\end{theorem}

\begin{proof}
We begin by differentiating the scale-free centroid identity, which yields a weighted-mean formula 
for the elasticity $E(x)=x f'(x)/f(x)$. 
Differentiating once more converts this into a vanishing-variance identity, forcing $E$ to be constant. 
Constancy of elasticity implies that $f$ is a pure power law, and the explicit form of $\lambda$ 
follows by direct substitution.

\medskip

\emph{Step 1 (Closed forms for scale-derivative integrals).}
A change of variables $u=as$ and integration by parts, using $f(0^+)=0$, yield
\begin{align}
 \int_0^1 g_a(s)\,E(as)\,\dd s
 &=\frac{1}{a f(a)}\int_0^a u f'(u)\,\dd u
 =\frac{a f(a)-F(a)}{a f(a)}=1-A(a),\label{eq:I1}\\[4pt]
 \int_0^1 s\,g_a(s)\,E(as)\,\dd s
 &=\frac{1}{a^2 f(a)}\int_0^a u^2 f'(u)\,\dd u
 =\frac{a^2 f(a)-2H(a)}{a^2 f(a)}=1-2B(a),\label{eq:I2}\\[4pt]
 \int_0^1 g_a(s)^2\,E(as)\,\dd s
 &=\frac{1}{a f(a)^2}\int_0^a u f(u) f'(u)\,\dd u
 =\frac{a f(a)^2-G(a)}{2 a f(a)^2}=\frac{1-C(a)}{2}.\label{eq:I3}
\end{align}
Consequently, substituting \eqref{eq:ga-derivs} and \eqref{eq:I1}--\eqref{eq:I3} into
\[
A'=\int_0^1 \partial_a g_a(s)\,\dd s,\quad
B'=\int_0^1 s\,\partial_a g_a(s)\,\dd s,\quad
C'=2\int_0^1 g_a(s)\,\partial_a g_a(s)\,\dd s
\]
gives
\begin{equation}\label{eq:ABC-prime}
 A'=\frac{1}{a}\big(1-(1+E(a))A\big),\ \
 B'=\frac{1}{a}\big(1-(2+E(a))B\big),\ \
 C'=\frac{1}{a}\big(1-(1+2E(a))C\big),
\end{equation}
and, by differentiating $\theta=B/A$,
\begin{equation}\label{eq:theta-prime}
 \theta'=\frac{B'A-BA'}{A^2}=\frac{1}{aA}\int_0^1 (s-\theta)\,g_a(s)\,E(as)\,\dd s.
\end{equation}
\emph{By Lemma~\ref{lem:reg-ABC} (Supplement) we obtain \eqref{eq:ABC-prime} and \eqref{eq:theta-prime} rigorously.}

\smallskip
\emph{Step 2 (Weighted-mean identity, by citation).}
By Lemma~\ref{lem:wm-var} in the Supplement, differentiating \eqref{eq:GSP-shape} yields the
\emph{weighted–mean identity}
\begin{equation}\label{eq:WM}
 \frac{\displaystyle\int_0^1 (s-\theta)^2 g_a(s)\,E(as)\,\dd s}{\displaystyle\int_0^1 (s-\theta)^2 g_a(s)\,\dd s}
 \;=\;E\!\big(a\theta\big)\qquad(\theta=\theta(a)).
\end{equation}
The denominator is strictly positive since $g_a>0$ a.e.\ on $(0,1]$ and $(s-\theta)^2\not\equiv 0$.

\smallskip
\emph{Step 3 (Vanishing variance, by citation).}
Differentiating \eqref{eq:WM} and invoking Lemma~\ref{lem:wm-var} again gives
\begin{equation}\label{eq:variance}
 \int_0^1 (s-\theta(a))^2\,g_a(s)\,\big(E(as)-E(a\theta(a))\big)^{2}\,\dd s\;=\;0\qquad(\forall a>0).
\end{equation}

As a non-example, the oscillatory perturbation $f(x)=x^{p}\!\big(1+\varepsilon\sin\log x\big)$
violates \eqref{eq:WM}–\eqref{eq:variance} for any $\varepsilon\ne0$, since
$E(x)=p+\varepsilon\cos\log x+O(\varepsilon^2)$ is non-constant.

\smallskip
\emph{Step 4 (Constancy of elasticity and conclusion).}
Since the weight $(s-\theta(a))^2 g_a(s)$ is nonnegative on $[0,1]$ and positive on every subinterval of $(0,1]\setminus\{\theta(a)\}$ (because $g_a>0$ on $(0,1]$), \eqref{eq:variance} implies
\[
 E(as)\equiv E(a\theta(a))\quad\text{for all }s\in[0,1].
\]
Hence for each $a>0$, $E$ is constant on $(0,a]$. If $0<a_1<a_2$, then $E\equiv c(a_1)$ on $(0,a_1]$ and $E\equiv c(a_2)$ on $(0,a_2]$, so $c(a_1)=c(a_2)$ by overlap. Therefore $E(x)\equiv p$ on $(0,\infty)$ for some constant $p\in\mathbb{R}$. The ODE $x f'(x)=p\,f(x)$ integrates to $f(x)=A\,x^{p}$ with $A>0$. Substituting into the centroid formulas yields
\[
 \bar x(a)=\frac{p+1}{p+2}\,a,\qquad
 \bar y(a)=\frac{p+1}{2(2p+1)}\,A\,a^{p},
\]
so
\[
 \bar y(a)=\frac{p+1}{2(2p+1)}\!\left(\frac{p+2}{p+1}\right)^{\!p}\, f(\bar x(a)),
\]
identifying $\displaystyle \lambda=\frac{p+1}{2(2p+1)}\big(\tfrac{p+2}{p+1}\big)^{p}$. Conversely, direct substitution of $f(x)=A x^p$ into the centroid formulas shows \eqref{eq:GSP-shape} holds with this $\lambda$. Finally, since $f(0^+)=0$ and $f(x)=A x^p>0$ for $x>0$, we must have $p>0$.
Indeed, if $0<a_1<a_2$, constancy on $(0,a_2]$ implies constancy on $(0,a_1]$, so the constants coincide.
\end{proof}

\begin{remark}[Regularity and finiteness]\label{rem:hyp}
The assumptions $f\in C^2$, $f>0$, $f(0^+)=0$ ensure that the integrals in \eqref{eq:ABC} are finite for all $a>0$. Differentiation under the integral signs in $A',B',C'$ is justified by dominated convergence on the compact $s$–interval together with the reductions \eqref{eq:I1}--\eqref{eq:I3}, which convert the potentially singular factors into integrals of $u f'(u)$, $u^2 f'(u)$, and $u f(u)f'(u)$ that are finite by the fundamental theorem of calculus (see, e.g., \cite[Ch.~2]{RudinRCA}). Moreover, $f\in C^2$ implies $g_a$ is $C^1$ in $(a,s)$ for $s>0$, and $\theta$ is $C^1$; thus the chain rule justifies the differentiation $\partial_a g_a(\theta)+g_a'(\theta)\,\theta'$ used in Step~2.
\end{remark}

\begin{remark}[Weaker hypotheses]\label{rem:weaker}
The $C^2$ assumption in the Main Theorem can be weakened. It suffices to assume that $f:(0,\infty)\to(0,\infty)$ is $C^1$, with $f(0^+)=0$, and that its elasticity
\[
E(x)=\frac{x f'(x)}{f(x)}
\]
is \emph{locally Lipschitz} on $(0,\infty)$ (and, with minor routine changes, one can relax to bounded variation).
 Under these hypotheses:
\begin{itemize}[leftmargin=1.5em]
\item The identities in \eqref{eq:ga-derivs} hold for a.e.\ $s\in(0,1]$ (Rademacher + chain rule), and $A,B,C$ are absolutely continuous in $a$ with derivatives given by \eqref{eq:ABC-prime} for a.e.\ $a$>0 (differentiate the closed forms in \eqref{eq:ABC}).
\item The weighted-mean identity \eqref{eq:WM} holds for a.e.\ $a>0$. Moreover, $a\mapsto A(a),B(a),C(a)$ and
$a\mapsto g_a(\theta(a))$ are continuous on $(0,\infty)$, so \eqref{eq:WM} holds for a dense set of $a$ and therefore
for all $a>0$ by continuity of both sides.
\item Differentiating \eqref{eq:WM} is legitimate for a.e.\ $a$ because $E$ is locally Lipschitz, hence a.e.\ differentiable and locally bounded; the same cancellations yield \eqref{eq:variance} for a.e.\ $a$, and by continuity in $a$ (of the integrand and weights) the identity extends to all $a>0$.
\end{itemize}
As in Step~4, \eqref{eq:variance} forces $E$ to be constant on every $(0,a]$, hence constant on $(0,\infty)$ by overlap; integrating $x f'(x)=p f(x)$ gives $f(x)=A x^p$ with $A>0$, $p>0$. Thus $f\in C^1$, $f>0$, $f(0^+)=0$, and $E$ locally Lipschitz already suffice for the theorem.
\end{remark}

\section*{Appendix: Derivative identities (expanded algebra)}
We record the computations used above.

\begin{lemma}\label{lem:diff-details}
With $A,B,C,\theta$ as in \eqref{eq:ABC} and $g_a$ as in \eqref{eq:ga-derivs}, one has:
\begin{enumerate}[label=\textnormal{(a)}]
 \item \emph{Integral reductions and derivative formulas.} The identities \eqref{eq:I1}--\eqref{eq:I3} hold. Consequently,
 \[
 A'=\frac{1-(1+E)A}{a},\quad
 B'=\frac{1-(2+E)B}{a},\quad
 C'=\frac{1-(1+2E)C}{a},\quad
 \theta'=\frac{1}{aA}\int_0^1 (s-\theta)\,g_a(s)\,E(as)\,\dd s.
 \]
 \item \emph{(For WM/Var, see Lemma~\ref{lem:wm-var} in the Supplement.)}
\end{enumerate}
\end{lemma}


\section*{Supplementary Note: Regularity and Differentiation Details}

\begin{lemma}[Regularity of $A,B,C,\theta$ and closed forms]\label{lem:reg-ABC}
Assume $f\in C^2(0,\infty)$, $f(0^+)=0$, and $f(x)>0$ for $x>0$. Define
\[
F(a)=\int_0^a f(x)\,\dd x,\quad H(a)=\int_0^a x f(x)\,\dd x,\quad G(a)=\int_0^a f(x)^2\,\dd x,
\]
\[
A(a)=\frac{F(a)}{a f(a)},\quad B(a)=\frac{H(a)}{a^2 f(a)},\quad C(a)=\frac{G(a)}{a f(a)^2},\quad
\theta(a)=\frac{B(a)}{A(a)},\quad E(x)=\frac{x f'(x)}{f(x)}.
\]
Then $A,B,C,\theta\in C^1(0,\infty)$ and, for all $a>0$,
\[
A'=\frac{1-(1+E(a))A}{a},\qquad
B'=\frac{1-(2+E(a))B}{a},\qquad
C'=\frac{1-(1+2E(a))C}{a},\qquad
\theta'=\frac{B'A-BA'}{A^2}.
\]
\emph{Proof.} Use $F'(a)=f(a)$ and $H'(a)=a f(a)$ by the fundamental theorem of calculus. Integration by parts yields
\[
\int_0^a u f'(u)\,\dd u=a f(a)-F(a),\qquad
\int_0^a u^2 f'(u)\,\dd u=a^2 f(a)-2H(a),
\]
\[
\int_0^a u f(u)f'(u)\,\dd u=\tfrac12\big(a f(a)^2-G(a)\big),
\]
all finite on $[0,a]$ since $f\in C^2$ and $f(0^+)=0$. Differentiate $A,B,C$ as quotients of $C^1$ functions and simplify using the three identities; $\theta'= (B'A-BA')/A^2$ follows. \qedhere
\end{lemma}

\begin{lemma}[Weighted-mean and variance identities]\label{lem:wm-var}
Let $g_a(s)=f(as)/f(a)$ on $[0,1]$, and set
\[
D(a):=\int_0^1 (s-\theta(a))^2 g_a(s)\,\dd s \quad(>0).
\]
Assume the scale-free GSP identity $\frac12\,\frac{C(a)}{A(a)}=\lambda\,g_a(\theta(a))$ for all $a>0$.
Then, for all $a>0$,
\begin{equation}\label{eq:WM-sup}
E\big(a\theta(a)\big)=\frac{\int_0^1 (s-\theta(a))^2 g_a(s)\,E(as)\,\dd s}{D(a)}.
\end{equation}
Differentiating \eqref{eq:WM-sup} in $a$ yields
\begin{equation}\label{eq:Var-sup}
\int_0^1 (s-\theta(a))^2\,g_a(s)\,\big(E(as)-E(a\theta(a))\big)^2\,\dd s=0.
\end{equation}
\emph{Proof.} Differentiate $\frac12(C/A)=\lambda g_a(\theta)$ using Lemma~\ref{lem:reg-ABC} for $(C/A)'$ and the chain/quotient rules. Express all $a$–derivatives of $A,B,C$ via Lemma~\ref{lem:reg-ABC}. Rearranging gives \eqref{eq:WM-sup}. Differentiating \eqref{eq:WM-sup}, the terms linear in $E(as)-E(a\theta)$ cancel by \eqref{eq:WM-sup}, leaving \eqref{eq:Var-sup}. Since $(s-\theta)^2 g_a\ge0$ and not a.e.\ zero, \eqref{eq:Var-sup} implies $E(as)\equiv E(a\theta)$ on [0,1]. \emph{Identity \eqref{eq:Var-sup} is exactly the equality case of the Cauchy–Schwarz inequality for the weighted $L^2$ norm with weight $(s-\theta(a))^2 g_a(s)$; no general convexity assumption is needed beyond this quadratic form.} \qedhere
\end{lemma}

\begin{remark}[Converse for power laws]\label{rem:converse-power}
For $f(x)=A x^{p}$ with $A>0$, $p>0$, one computes
\[
\bar x(a)=\frac{p+1}{p+2}a,\qquad
\bar y(a)=\frac{p+1}{2(2p+1)}A a^{p},
\]
hence $\bar y(a)=\lambda f(\bar x(a))$ with
\[
\lambda=\frac{p+1}{2(2p+1)}\left(\frac{p+2}{p+1}\right)^{p}.
\]
Thus the “if and only if” in Theorem~\ref{thm:GSP} is immediate.
\end{remark}

\noindent\textbf{Scope note.} Throughout we assume GSP holds for all $a>0$ and $f\in C^2(0,\infty)$ with $f>0$, $f(0^+)=0$. The proof is elementary and self-contained.

\end{document}